\documentclass[12pt]{amsart}            
\usepackage{latexsym,mathtools}
\usepackage{epsfig,setspace}
\usepackage{a4}
\usepackage{amssymb}
\usepackage{amsthm}
\usepackage{amsbsy} 
\usepackage{upgreek}
\usepackage{relsize}
\usepackage{tikz-cd}
\usepackage{amsfonts,amssymb,latexsym,amscd,amsmath,euscript,enumerate,verbatim,calc}
\allowdisplaybreaks[1]
\usepackage{url}
\usepackage{hhline}
\usepackage{pifont}
\usepackage{relsize}
\usepackage[utf8]{inputenc}
\usepackage[T1]{fontenc}
\usepackage[colorlinks=true,linkcolor=blue,citecolor=magenta]{hyperref} 

\theoremstyle{definition}

\newtheorem{theorem}{Theorem}[section]
\newtheorem{lemma}[theorem]{Lemma}
\newtheorem{proposition}[theorem]{Proposition}
\newtheorem{question}[theorem]{Question}

\newtheorem{corollary}[theorem]{Corollary}
\newtheorem{definition}[theorem]{Definition}

\newtheorem{remark}[theorem]{Remark}

\def\F{{\mathbb F}}

\def\Fq{{\mathbb F}_q}

\def \mFq {\mathbb{F}_{q}}

\def \bzero {{\bf 0}}
\def \bfb {{\bf b}}

\def \calP {{\mathcal P}} 
\def \calM {{\mathcal C}}

\newcommand{\diag}{\operatorname{diag}}

\newcommand{\rank}{\operatorname{rank}}
\newcommand{\GL}{\operatorname{GL}}

\makeatletter
\def\imod#1{\allowbreak\mkern10mu({\operator@font mod}\,\,#1)}
\makeatother

\title{Unimodular Polynomial Matrices over Finite Fields} 

\author{Akansha Arora}
\address{Indraprastha Institute of Information Technology Delhi (IIIT-Delhi), New Delhi 110020, India.}
\email{akanshaa@iiitd.ac.in}

\author{Samrith Ram}
\address{Indraprastha Institute of Information Technology Delhi (IIIT-Delhi), New Delhi 110020, India.}
\email{samrith@gmail.com}

\author{Ayineedi Venkateswarlu}
\address{Computer Science Unit, Indian Statistical Institute - Chennai Centre, Chennai 600029, India.}
\email{venku@isichennai.res.in}



\begin{document}
\begin{abstract}
We consider some combinatorial problems on matrix polynomials over finite fields. Using results from control theory we give a proof of a result of Lieb, Jordan and Helmke on the number of linear unimodular matrix polynomials over a finite field. As an application of our results we give a new proof of a theorem of Chen and Tseng which answers a question of Niederreiter on splitting subspaces. We use our results to affirmatively resolve a conjecture on the probability that a matrix polynomial is unimodular.   
\end{abstract}
\maketitle
\tableofcontents 
\section{Introduction}  
Denote by $\Fq$ the finite field with $q$ elements where $q$ is a prime power. Let $\Fq[x]$ denote the ring of polynomials over $\Fq$ in the indeterminate $x$. For any ring $R$ and positive integers $n,k$ define $M_{n,k}(R)$ to be the set of all $n\times k$ matrices over $R$. Similarly $M_k(R)$ denotes the ring of $k\times k$ matrices over $R$. Denote by $I_{n,k}$ the matrix in $M_{n,k}(\Fq)$ whose $(i,j)$\textsuperscript{th} entry is zero whenever $i\neq j$ and equal to 1 for $i=j$.

The main objects of study in this paper are matrix polynomials over finite fields. A matrix polynomial over a field $F$ in the variable $x$ is a sum $\sum_{i=0}^{d}A_ix^i$, where $A_i\in M_{n,k}(F)(0\leq i\leq d)$ for some fixed positive integers $n,k$. It is often convenient to view such a matrix polynomial as a single matrix whose entries are polynomials in $x$ (sometimes referred to as a polynomial matrix) and we freely alternate between these two points of view. A matrix polynomial ${\bf A}=\sum_{i=0}^{d}A_ix^i\in M_{n,k}(\Fq[x])$ is \emph{unimodular} if the greatest common divisor of all $r\times r$ minors of ${\bf A}$ is equal to 1 where $r=\min\{n,k\}$. The notion of unimodularity can be defined more generally for rectangular matrices over an arbitrary integral domain. A landmark result in the setting of unimodularity is the Quillen-Suslin theorem \cite{MR0427303,MR0469905} formerly known as Serre's conjecture. We refer to \cite{MR3008525,JL2018,MR2763589,MR3706911}  for other contexts where unimodularity is considered. We begin with a combinatorial question concerning matrix polynomials over a finite field.  
\begin{question}\label{q:first}   
  Given positive integers $n,k$ and a prime power $q$, determine the number of matrices $A\in M_{n,k}(\Fq)$ for which the matrix polynomial $xI_{n,k}-A$ is unimodular.
\end{question}    
    
This question was essentially considered by Koci\k{e}cki and Przyłuski \cite{MR1019984} (also see \cite[Prob. 1.2]{zerokernel}) in an attempt to determine the number of reachable pairs of matrices over a finite field. Reachability is a fundamental notion in the control theory of linear systems. The question was fully answered only recently by Lieb, Jordan and Helmke \cite[Thm.~1]{HJL2016} who showed that the answer is equal to $\prod_{i=1}^{k}(q^n-q^i)$. 
We refer to the introduction of \cite{zerokernel} for details and alternate formulations of the result of Lieb et al. Our main result is Lemma \ref{eT} which allows us to give a new proof (Corollary \ref{cor:numsimp}) of the theorem of Lieb et al. An essential ingredient in our main lemma is a control theoretic result of Brunovský on completely controllable pairs. 

Further applications of our results appear in Sections \ref{sec:split} and \ref{sec:unimodular}. In Section \ref{sec:split} we consider splitting subspaces (defined below) which were introduced by Niederreiter \cite[Def. 1]{N2} in the context of his work on the multiple recursive matrix method for pseudorandom number generation. 
\begin{definition}\label{def:splittingsubspace}
  Let $d,m$ be positive integers and consider the vector space $\F_{q^{md}}$ over $\Fq$. 
  For any element $\alpha\in \F_{q^{md}}$ an $m$-dimensional subspace $W$ of $\F_{q^{md}}$ is $\alpha$-\emph{splitting} if
  \begin{align*}
    \F_{q^{md}}=W\oplus \alpha W\oplus \cdots \oplus \alpha^{d-1}W.
  \end{align*}
\end{definition}
 Niederreiter was interested in the following question on splitting subspaces.
\begin{question}
  Given $\alpha\in \F_{q^{md}}$ such that $\F_{q^{md}}=\Fq(\alpha)$, what is the number of $\alpha$-splitting subspaces of $\F_{q^{md}}$ of dimension $m$?
\end{question} 
It may be noted that the same question was also considered by Goresky and Klapper (see the remark in ~\cite[p. 1653]{GoreskyKlapper2006} and~\cite[Thm. 3(4)]{GoreskyKlapper2006}). In addition to the evident cryptographic aspect, Niederreiter's question also has interesting connections with group theory and finite projective geometry via block companion Singer cycles. We refer to \cite{GSM,m=2}  for more on this topic. The case $m=2$ of Niederreiter's question was settled in \cite{m=2} using a result that answers the following question: What is the probability that two randomly chosen polynomials of a fixed positive degree over a finite field are coprime? This question on the probability of coprime polynomials goes back to an exercise in Knuth \cite[\S 4.6.1, Ex.~5]{K} and has subsequently been considered by Corteel, Savage, Wilf and Zeilberger \cite{MR1620873} in the more general setting of combinatorial prefabs. Further results on the degree distribution of the greatest common divisor of random polynomials over a finite field appear in \cite{MR2238027}. In fact, our main result relies on Lemma~\ref{rprime} which may be viewed as a probabilistic result on coprime polynomials. Chen and Tseng \cite[Cor.~3.4]{sscffa} eventually answered Niederreiter's question on splitting subspaces by proving the following theorem which was initially conjectured in \cite[Conj. 5.5]{m=2}.   
\begin{theorem}[Splitting Subspace Theorem]
\label{th:numsplitting}
  For any $\alpha\in \F_{q^{md}}$ such that $\F_{q^{md}}=\Fq(\alpha)$, the number of $\alpha$-splitting subspaces of $\F_{q^{md}}$ of dimension $m$ is precisely
  $$
\frac{q^{md}-1}{q^m-1}q^{m(m-1)(d-1)}.
  $$
\end{theorem}
  
In this paper a control-theoretic result of Wimmer (Theorem \ref{th:wimmer}) is used to prove Theorem~\ref{th:fibersize} from which the Splitting Subspace Theorem follows as a corollary. 
In Section \ref{sec:unimodular} a generalization of Question \ref{q:first} is considered. The answer to this question which was stated earlier can be given a probabilistic flavour as follows.
\begin{theorem}\label{th:linearunimodular}  
If a matrix $A$ is selected uniformly at random from $M_{n,k}(\Fq)$, then the probability that $xI_{n,k}-A$ is unimodular is given by
$\prod_{i=1}^k (1-q^{i-n})$.  
\end{theorem}
Using results in Section \ref{sec:simple}, we prove a conjecture (Theorem~\ref{th:density}) proposed in \cite{zerokernel} on the proportion of unimodular polynomial matrices which generalizes Theorem~\ref{th:linearunimodular}.


\section{Simple Linear Transformations}\label{sec:simple}
We begin by recalling the notion of a simple linear transformation \cite[Def.~3.1]{zerokernel}.  

\begin{definition}
\label{def:simple}
   Let $V$ denote a vector space over a field $F$ and let $W$ be a subspace of $V$. An $F$-linear transformation $T:W\to V$ is \emph{simple} if the only $T$-invariant subspace properly contained in $V$ is the zero subspace.
 \end{definition}

 \begin{remark}
Note that the definition requires that there are no $T$-invariant subspaces properly contained in $V$ rather than in $W$. The reason being that if $W$ is a proper subspace, then the definition does not allow $W$ itself to be $T$-invariant. In the case $W=V$ we necessarily have that $W$ is $T$-invariant. It can be shown that a linear operator $T$ on a finite dimensional vector space $V$ is simple if and only if it has an irreducible characteristic polynomial. In fact simple maps defined on a proper subspace $W$ of a vector space $V$ are precisely the restrictions to $W$ of simple maps defined on all of $V$.
   
 \end{remark}

 The following proposition 
 elucidates the connection between simple linear transformations and unimodularity.
 \begin{proposition}\label{pr:simplematrix}
   Let $V$ be an $n$-dimensional vector space over $F$ with ordered basis $\mathcal{B}_n=\{v_1,\ldots,v_n\}$. Let $\mathcal{B}_k=\{v_1,\ldots,v_k\}$ denote the ordered basis for the subspace $W$ spanned by $v_1,\ldots,v_k$. Let $T:W\to V$ be a linear transformation and let $Y\in M_{n,k}(F)$ denote the matrix of $T$ with respect to $\mathcal{B}_k$ and $\mathcal{B}_n$. Then $T$ is simple if and only if $xI_{n,k}-Y$ is unimodular.
\end{proposition}
\begin{proof}
  See \cite[Prop. 2.5]{zerokernel} and \cite[Prop. 3.2]{zerokernel}.
\end{proof}

Let $m$ be a positive integer and let ${\bf a} = (a_1,\ldots,a_m)\in \mFq^m$ be an arbitrary but fixed nonzero vector. Let $t$ be the largest index such that $a_t\neq 0$.  Let $d_1\geq d_2\geq \cdots \geq  d_m$
be a nonincreasing sequence of integers with $d_t \geq -1$.
Let $N_{\bf a}(d_1,\ldots d_m)$ denote the number of $m$-tuples $(f_1,\ldots,f_m)$ of polynomials over $\Fq$
such that $f_i = a_i x^{d_i+1} + h_i$ and $\deg h_i\leq d_i$ for $1\leq i\leq m$ with $\gcd(f_1,\ldots,f_m)=1$. Here we interpret negative powers of $x$ to be zero. Since $a_t\neq 0$, we necessarily have $\deg f_t = d_t+1$ for any tuple $(f_1,\ldots,f_m)\in N_{\bf a}(d_1,\ldots,d_m)$. We adopt the convention that the degree of the zero polynomial is $-\infty$. 
Note that if there is some $s \geq t$ such that $d_i<0$ for each $s < i\leq m$, then  $N_{\bf a}(d_1,\ldots,d_m)=N_{\bf a'}(d_1,\ldots,d_s)$ where ${\bf a'}= (a_1,\ldots,a_s)$. 


We adapt an argument in the proof of ~\cite[Thm. 4.1]{Mario} to prove the following lemma which is central to our main result.
\begin{lemma}\label{rprime}  
Let $m$ be a positive integer and let $d_1\geq d_2\geq \cdots\geq d_m\geq 0$ be a sequence of integers. Let ${\bf a} = (a_1,\ldots,a_m)\in \mFq^m$ be a fixed nonzero vector. We have
\[N_{\bf a}(d_1,\ldots,d_m)=q^{k+m}-q^{k+1},\]
where $k=d_1+\cdots+d_m$. 
\end{lemma}
\begin{proof}
  Fix a positive integer $m$. 
Let $S(d_1,\ldots,d_m)$ denote the set of ordered $m$-tuples $(f_1,\ldots,f_m)$ where $f_i = a_i x^{d_i+1} + h_i$ for some $h_i$ with
$\deg h_i\leq d_i$ for $1\leq i\leq m$. Let $t$ be the largest index such that $a_t\neq 0$. We partition $S(d_1,\ldots,d_m)$ into disjoint subsets $S_0,S_1,\ldots,S_{d_t+1}$
where the set $S_d \;(0\leq d\leq d_t+1)$ denotes the set of $m$-tuples in $S(d_1,\ldots,d_m)$ whose GCD is a monic polynomial
of degree $d$. For each monic polynomial $h$ over $\Fq$ of degree $d$ and any coprime $m$-tuple $(g_1,\ldots,g_m)$ of polynomials in $S(d_1-d,\ldots,d_m-d)$, it is easy to see that $(g_1h,g_2h,\ldots,g_mh)\in S_d$. Conversely, for any tuple $(f_1,\ldots,f_m)\in S_d$, the polynomial $h=\gcd(f_1,\ldots,f_m)$ is monic of degree $d$ and $(f_1/h,\ldots,f_m/h)$ is an ordered $m$-tuple of coprime polynomials in $S(d_1-d,\ldots,d_m-d)$. As a result, we have $|S_d|=q^d N_{\bf a}(d_1-d,\ldots,d_m-d)$ for $0\leq d\leq d_t+1$.
For $k=d_1+\cdots+d_m$, we have
  \begin{equation}
    \label{eq:first}
q^{k+m}=\sum_{d=0}^{d_t+1}|S_d|=\sum_{d=0}^{d_t+1}q^dN_{\bf a}(d_1-d,\ldots,d_m-d).    
  \end{equation}

Replacing $d_i$ by $d_{i}+1$ for each $1\leq i\leq m$, we obtain
\begin{align*}
  q^{k+2m}&=\sum_{d=0}^{d_t+2}q^dN_{\bf a}(d_1+1-d,\ldots,d_m+1-d)  \\
          &=\sum_{d=-1}^{d_t+1}q^{d+1}N_{\bf a}(d_1-d,\ldots,d_m-d) \\
          &=N_{\bf a}(d_1+1,\ldots,d_m+1)+q\sum_{d=0}^{d_t+1}q^dN_{\bf a}(d_1-d,\ldots,d_m-d)\\
  &=N_{\bf a}(d_1+1,\ldots,d_m+1)+q(q^{k+m}),
\end{align*}
where the last equality follows from \eqref{eq:first}. It follows that $N_{\bf a}(d_1+1,\ldots,d_m+1)=q^{k+2m}(1-q^{1-m})$, or equivalently, $N_{\bf a}(d_1,\ldots,d_m)=q^{k+m}-q^{k+1}$ as desired.
\end{proof}

As the language of control theory is used in the proof of our main result we collate here a few definitions \cite[IX.2]{Partially} and results that are referred to later on.
  In what follows, $F$ denotes an arbitrary field and $k,\ell$ are fixed positive integers.
\begin{definition}
A matrix pair $(A,B)\in M_{k,k}(F)\times M_{k,\ell}(F)$ is a \emph{reachable pair}  if the $k\times k\ell$ matrix $S(A,B):=   \begin{bmatrix}
B &AB &\cdots & A^{k-1}B
  \end{bmatrix}$ has rank equal to $k$.
  
\end{definition}
\begin{remark}
  \label{rem:reachableifunimodular}
A pair $(A,B)$ is reachable if and only if the polynomial matrix $[xI_k-A \; B]$ is unimodular.  
\end{remark}

\begin{definition}
Associate with each pair $(A,B)\in M_{k,k}(F)\times M_{k,\ell}(F)$ a sequence of integers $p_i(i\geq 1)$ by defining $p_1:=\mathrm{rank}\; B$ and for $i\geq 2$, 
  $$
p_i:= \rank  \begin{bmatrix}
B &AB &\cdots & A^{i-1}B
\end{bmatrix}
- \rank  \begin{bmatrix}
B &AB &\cdots & A^{i-2}B
  \end{bmatrix}.
  $$
  Consider the dual sequence $k_j(j\geq 1)$ defined by $k_j=\#\{r:p_r\geq j\}.$ The numbers $k_1,\ldots,k_\ell$ are called the \emph{controllability indices} 
  of the pair $(A,B)$.
  
\end{definition}

For any positive integer $m$, denote by $\GL_m(F)$ the general linear group of $m\times m$ nonsingular matrices over $F$.
Define \cite[P. 3]{ZabInv}
$$
\Gamma_{k,\ell}:=\left\{ \begin{bmatrix}
    P & {\bf 0}\\
    R & Q
  \end{bmatrix}\in \GL_{k+\ell}(F): P\in \GL_k(F), Q\in \GL_\ell(F), R\in M_{\ell,k}(F)\right\}.
$$
\begin{definition}
  Two pairs $(A_1,B_1)$ and $(A_2,B_2)$ in $ M_{k,k}(F)\times M_{k,\ell}(F)$ are said to be $\Gamma_{k,\ell}$-equivalent \cite[Def. 2.1]{ZabInv} if there exists a matrix $P\in \Gamma_{k,\ell}$ such that for each pair of matrices $C_1\in M_{\ell,k}(F)$ and $D_1\in M_\ell(F)$, there exist matrices $C_2\in M_{\ell,k}(F)$ and $D_2\in M_{\ell}(F)$ such that
  $$
P \begin{bmatrix}
    A_1 & B_1\\
    C_1 & D_1
  \end{bmatrix} P^{-1}=\begin{bmatrix}
    A_2 & B_2\\
    C_2 & D_2
  \end{bmatrix}.
  $$
\end{definition}
When the values of $k,\ell$ are clear from the context, we refer to $\Gamma_{k,\ell}$-equivalence simply as $\Gamma$-equivalence. The following result (\cite{Brunovsk1970ACO}, \cite[Lem. 2.7]{ZabInv}) is due to Brunovsky.
\begin{theorem}
  \label{th:Brunovsky}
  Let $(A,B)\in M_{k,k}(F)\times M_{k,\ell}(F)$. Suppose $(A,B)$ is a reachable pair with $\rank B=r$ and $k_1\geq \cdots\geq k_r>k_{r+1}=\cdots=k_\ell(=0)$ are the controllability indices of $(A,B)$. Then $(A,B)$ is $\Gamma$-equivalent to a pair $(A_c,B_c)\in  M_{k,k}(F)\times M_{k,\ell}(F)$ of the following form:
  \begin{enumerate}[i)]
  \item $A_c$ is the block diagonal matrix $\diag (A_1,\ldots,A_r)$ where $A_i$ is the $k_i\times k_i$ matrix
    $$
\begin{bmatrix}
  {\bf 0} & I_{k_i-1}\\
  0 & {\bf 0}
  \end{bmatrix};
  $$
\item $B_c$ is of the block form $[ B'\;  {\bf 0}]$, where $B'$ denotes the $k\times r$ matrix
  $$
B'=\begin{bmatrix}
    E_1 \\
    \vdots\\
    E_r
  \end{bmatrix}; \mbox{ where } E_i=\begin{bmatrix}
    {\bf 0}\\
   e_i
  \end{bmatrix} \in M_{k_i\times r}(F),
  $$
  and $e_i$ denotes the $i$\textsuperscript{th} row of the $r\times r$ identity matrix.
  \end{enumerate}
\end{theorem}

The following lemma is our main result.
\begin{lemma}\label{eT}
Let $n,k$ be integers with $0\leq k<n-1$. Let $V$ be an $n$-dimensional vector space over $\Fq$ and let $W,W'$ be fixed subspaces of $V$ of dimensions $k$ and $k+1$ respectively with $W\subset W'$. Suppose $T:W\to V$ is a simple linear transformation. Then the number of simple linear transformations $T':W'\to V$ such that $T'_{\mid W}=T$ (the restriction of $T'$ to $W$ is $T$) is equal to $q^n-q^{k+1}$.
\end{lemma}
\begin{proof}
First suppose $k=0$. In this case $W'$ is spanned by some nonzero vector $w\in V$. Then $T'$ is simple precisely
when $T'(w)$ does not lie in the span of $w$. So the number of such linear transformations is clearly $q^n-q$. 

Suppose $k\geq 1$. Let $\mathcal{B}_k = \{v_1,\ldots,v_k \}$ be an ordered basis for $W$ and
$\mathcal{B}_n = \{v_1,\ldots,v_n \}$ be an ordered basis for $V$ obtained by extending $\mathcal{B}_k$.
Let $Y$ be the matrix of $T$ with respect to $\mathcal{B}_k$ and $\mathcal{B}_n$.
Since $T$ is simple, $ {\bf Y} = x I_{n,k} - Y$ is unimodular by Proposition~\ref{pr:simplematrix}.
Suppose that $Y = \left[\!\!\begin{array}{c} A\\ C\\ \end{array}\!\!\right]$ for some $A \in M_{k,k}(\mFq)$
and $C\in M_{n-k,k}(\mFq)$. Since $ {\bf Y} = x I_{n,k} - Y $ is unimodular, it follows by Remark~\ref{rem:reachableifunimodular} that $(A^t,C^t)$ is a reachable pair. 
Suppose that $\mbox{rank}(C) = r$, and $k_1\ge k_2\ge \cdots \ge k_r > k_{r+1} = \cdots = k_{n-k} = 0$
are the controllability indices of the pair $(A^t,C^t)$.   
We have $k_1 + \cdots + k_r = k$. By Theorem \ref{th:Brunovsky} 
we may assume that $A$ and $C$ are of the following form:

$A = \mbox{diag}(A_1,A_2,\ldots,A_r)$, where $A_i$ is the $k_i\times k_i$ matrix
$\left[\begin{array}{cc} \bzero & 0\\ I_{k_i-1} & \bzero\\ \end{array}\right];$\\
\indent$C = \left[\!\!\begin{array}{c} C'\\ \bzero\\ \end{array}\!\!\right]$, where
$ C' = \left[E_1\,\cdots\,E_r\right] \in M_{r,k}(\mFq)$ with 
$E_i = [\bzero\,e_i] \in M_{r,k_i}(\mFq),$
and $e_i$ denotes the $i$\textsuperscript{th} column of the $r\times r$ identity matrix for $1\leq i\leq r$.
Let $\lambda_s = \sum_{i=1}^s k_i$ for $1\le s\le r$ and set $\lambda_0=0$.
Then the linear transformation $T$ can be described by
\begin{equation}\label{eq:T}
T(v_j) =
\begin{cases}
v_{k+s} &\ \mbox{if}\ j = \lambda_s\ \mbox{for some}\ s,\, 1\le s\le r;\\
v_{j+1} &\  \mbox{otherwise,}
\end{cases}
\end{equation}
where $1\le j\le k$.
Also the matrix $Y$ can be described by
\[ Y = [{\bf e}_{\lambda_0+2},\ldots,{\bf e}_{\lambda_1},{\bf e}_{k+1}, {\bf e}_{\lambda_1+2},\ldots,{\bf e}_{\lambda_2},{\bf e}_{k+2},
 \ldots, {\bf e}_{\lambda_{r-1}+2},\ldots,{\bf e}_{\lambda_r},{\bf e}_{k+r}],\]
where ${\bf e}_i$ is the $i$\textsuperscript{th} column of the identity matrix $I_n$.
Let $U = \mathrm{span}(\mathcal{B}_n\setminus\mathcal{B}_k)$ be the subspace of $V$
spanned by $\{v_{k+1},\ldots,v_n\}$. We have $V = W\oplus U$.

Now $W\subset W'$ and $W'$ is of dimension $k+1$. Since $V = W\oplus U$, there is a nonzero vector $w\in W' \cap U$. Let $\{v_{k+1}'=w,v_{k+2}',\ldots,v_n'\}$ be an ordered basis for $U$.
Since $V = W\oplus U$, we have $\mathcal{B}_n' = \{v_1,\ldots,v_k,v_{k+1}',\ldots,v_n'\}$ is an ordered basis for $V$.
Let $R$ be the matrix of the identity map $1_V$ on $V$ with respect to the bases $\mathcal{B}_n'$ and $\mathcal{B}_n$.
Note that the matrix $R$ can be expressed as
\[\left[\begin{array}{cc} I_k & \bzero\\ \bzero & S\\ \end{array}\right],\]
where $S$ is the matrix of the identity map $1_U$ on $U$ with respect to the bases
$\mathcal{B}_n'\setminus \mathcal{B}_k$ and $\mathcal{B}_n\setminus \mathcal{B}_k$.
Let $v_{k+1}' = \sum_{j=k+1}^{n} c_j v_j$ for some scalars $c_j$. Then the first column of $S$
is given by $(c_{k+1},\ldots,c_n)^t\in \mFq^{n-k}$.
The matrix $\tilde{Y}$ of $T$ with respect to $\mathcal{B}_k$ and $\mathcal{B}_n'$ is given by
$\tilde{Y} = R^{-1}Y$. Define $\mathcal{B}_{k+1}' = \{v_1,\ldots,v_k,v_{k+1}' \}$ and let $Y'$ be the matrix of $T'$ with respect to the bases $\mathcal{B}_{k+1}'$ and
$\mathcal{B}_n'$. Since $T'_{\mid W}=T$ we have $Y' = R^{-1} [Y\,{\bf b}]$ for some column
vector ${\bf b}\in \mathbb{F}_q^n$.
By Proposition~\ref{pr:simplematrix}, $T'$ is simple if and only if ${\bf Y}' = x I_{n,k+1} - Y'$ is unimodular.
Let ${\bf Y_b} = R {\bf Y}' = R (x I_{n,k+1} - R^{-1} [Y\,{\bf b}]) = x R_{k+1} - [Y\,{\bf b}]$,
where $R_{k+1}$ is the submatrix formed by the first $(k+1)$ columns of $R$.
We have ${\bf Y_b} = [{\bf Y}\, x{\bf c}-{\bf b}]$, where
${\bf c} = (0,\ldots,0,c_{k+1},\ldots,c_n)^t\in \mFq^n$.

Suppose $\bfb = (b_1, b_2, \ldots, b_n)^t \in \mFq^n$. Then the matrix $Y_{\bf b} = [Y\ \bfb]$ is of the form
\begin{align}\label{matyp}
Y_{\bf b} = \left[\begin{array}{ccccc}
A_1 & \bzero & \ldots & \bzero & \bfb_1 \\
\bzero & A_2 & \ldots & \bzero & \bfb_2 \\
\vdots & \vdots & \ddots & \vdots & \vdots \\
\bzero & \bzero & \ldots & A_r & \bfb_r\\
E_1 & E_2 & \ldots & E_r & \tilde{\bfb}\\
\bzero & \bzero & \ldots & \bzero & \hat{\bfb}\\
\end{array}\right],
\end{align}
where $\bfb_i = (b_{\lambda_{i-1}+1},\ldots,b_{\lambda_{i}})^t \in \mFq^{k_i}$ for $1\le i\le r$,
$\tilde{\bfb} = (b_{k+1},\ldots,b_{k+r})^t\in \mFq^{r}$, and
$\hat{\bfb} = (b_{k+r+1},\ldots,b_{n})^t\in \mFq^{n-k-r}$.

Now consider the polynomial matrix ${\bf Y}_{\bf b} = [{\bf Y}\, x{\bf c}-{\bf b}]$.
We permute the rows of ${\bf Y}_{\bf b}$ in the following way: for each $1\le i\le r-1$, arrange the $(k+i)$\textsuperscript{th} row of ${\bf Y}_{\bf b}$
in between the $i$\textsuperscript{th} and $(i+1)$\textsuperscript{th} block rows appearing in~(\ref{matyp}).
The resulting matrix ${\bf Z}$ is of the following form:
\begin{align}\label{matz}
{\bf Z} = \left[\begin{array}{ccccc}
{\bf Z}_1 & \bzero & \ldots & \bzero & \bfb_1' \\
\bzero & {\bf Z}_2 & \ldots & \bzero & \bfb_2' \\
\vdots & \vdots & \ddots & \vdots & \vdots \\
\bzero & \bzero & \ldots & {\bf Z}_r & \bfb_r'\\
\bzero & \bzero & \ldots & \bzero & \bfb'\\
\end{array}\right],
\end{align} 
where ${\bf Z}_i = x \left[\!\!\begin{array}{c} I_{k_i}\\ \bzero\\ \end{array}\!\!\right] -
\left[\!\!\begin{array}{c} \bzero\\ I_{k_i}\\\end{array}\!\!\right]$,
$\bfb_i' = \left[\!\!\begin{array}{c} -\bfb_i\\ c_{k+i}x-b_{k+i}\\\end{array}\!\!\right]$ for $1\le i\le r \mbox{ and } \bfb' = (c_{k+r+1}x-b_{k+r+1},\ldots,c_{n}x-b_{n})^t$.
Now we apply the following sequence of elementary row operations to ${\bf Z}$ to eliminate $x$ in
the first $k$ columns: in the first block row appearing in~(\ref{matz}), add $x$ times the
$(i+1)$\textsuperscript{th} row to the $i$\textsuperscript{th} row successively for $i=k_1, k_1-1, \ldots, 1$
in that order. Similarly we apply elementary row
operations to the other block rows. By appropriate elementary column operations, the entries in the last
column can be made zero at suitable positions. Eventually we can transform the matrix to the following form:
\begin{align}\label{matzp}
{\bf Z}' = \left[\begin{array}{ccccc}
{\bf Z}'_1 & \bzero & \ldots & \bzero & \bfb_1'' \\
\bzero & {\bf Z}_2' & \ldots & \bzero & \bfb_2'' \\
\vdots & \vdots & \ddots & \vdots & \vdots \\
\bzero & \bzero & \ldots & {\bf Z}_r' & \bfb_r''\\
\bzero & \bzero & \ldots & \bzero & \bfb'\\
\end{array}\right],
\end{align}  
where ${\bf Z}_i' = -\left[\!\!\begin{array}{c} \bzero\\ I_{k_i}\\\end{array}\!\!\right]$,
$\bfb_i'' = \left[\!\!\begin{array}{c} f_i\\ \bzero\\\end{array}\!\!\right]$ with
$f_i(x) = c_{k+i} x^{k_i+1} - b_{k+i} x^{k_i} - \sum_{j=1}^{k_i} b_{\lambda_{i-1}+j} x^{j-1}$ for $1\le i\le r$
and $\bfb' = (c_{k+r+1}x-b_{k+r+1},\ldots,c_{n}x-b_{n})^t$.

Let $g = \gcd(f_1,f_2,\ldots,f_r, c_{k+r+1} x - b_{k+r+1},\ldots,c_n x - b_n)$.
The matrix ${\bf Z}'$ is unimodular if and only if $g=1$. By Lemma~\ref{rprime}
it follows that the number of vectors $\bfb \in \mFq^n$  such that $g = 1$ is given by $q^n- q^{k+1}$.
As ${\bf Y}'$ and ${\bf Z}'$ are equivalent, the result follows.

\end{proof}

The lemma can be recast in the setting of matrices as follows.

\begin{corollary}\label{eumm}
  Let $Y\in M_{n,k}(\Fq)$ be such that the linear matrix polynomial $xI_{n,k}-Y$ is unimodular. For each column vector ${\bf b}\in \F_{q}^n$ let $Y_{\bf b}=[Y \;{\bf b}]\in M_{n,k+1}(\Fq)$. 
  Then the number of column vectors ${\bf b}\in \F_{q}^n$ for which $xI_{n,k+1}-Y_{\bf b}$ is unimodular equals $q^n-q^{k+1}$.  
\end{corollary}

We can now give an alternate proof of \cite[Thm. 3.8]{zerokernel} concerning the number of simple linear transformations with a fixed domain.
\begin{corollary}
  Let $V$ be an $n$-dimensional vector space over $\Fq$ and $W$ be a proper $k$-dimensional subspace of $V$. The number of simple linear transformations $T:W\to V$ equals
  $$
\prod_{i=1}^{k}(q^n-q^i).
  $$
\end{corollary}
We may use Proposition \ref{pr:simplematrix} to reformulate the corollary in terms of matrices. This allows us to answer Question \ref{q:first} stated in the introduction.
\begin{corollary}\label{cor:numsimp}
  Let $n,k$ be positive integers with $k<n$. The number of matrices $A\in M_{n,k}(\Fq)$ such that $xI_{n,k}-A$ is unimodular equals
    $$
\prod_{i=1}^{k}(q^n-q^i).
  $$
\end{corollary}
By repeated application of Corollary \ref{eumm} we obtain the following extension which is used later on in Sections \ref{sec:split} and \ref{sec:unimodular}.
\begin{lemma}\label{lem:extension}
Let $n, k, t$ be positive integers such that $k+t<n$. Suppose that the matrix polynomial $xI_{n,k}-Y$ is unimodular for some $Y\in M_{n,k}(\Fq)$. The number of matrices $A\in M_{n,t}(\mFq)$
such that the matrix polynomial 
\[ x I_{n,k+t} - [Y \; A]\] 
 is unimodular is equal to $\prod_{i=1}^{t} (q^n - q^{k+i})$.
\end{lemma}

\section{Splitting Subspaces}\label{sec:split}
Recall the definition of splitting subspace given earlier in the introduction.
\begin{definition}
  Let $d,m$ be positive integers and consider the vector space $\F_{q^{md}}$ over $\Fq$. 
  For any element $\alpha\in \F_{q^{md}}$ an $m$-dimensional subspace $W$ of $\F_{q^{md}}$ is $\alpha$-\emph{splitting} if
  \begin{align*}
    \F_{q^{md}}=W\oplus \alpha W\oplus \cdots \oplus \alpha^{d-1}W.
  \end{align*}
\end{definition}

Closely related to splitting subspaces are block companion matrices which we define below.
\begin{definition}
For positive integers $m,d$, an $(m,d)$-block companion matrix over $\Fq$ is a matrix in $M_{md}(\Fq)$ of the form
\begin{equation}
\label{bcm} 
\begin {pmatrix}
\mathbf{0} & \mathbf{0} & \mathbf{0} & . & . & \mathbf{0} & \mathbf{0} & C_0\\
I_m & \mathbf{0} & \mathbf{0} & . & . & \mathbf{0} & \mathbf{0} & C_1\\
. & . & . & . & . & . & . & .\\
. & . & . & . & . & . & . & .\\
\mathbf{0} & \mathbf{0} & \mathbf{0} & . & . & I_m & \mathbf{0} & C_{d-2}\\
\mathbf{0} & \mathbf{0} & \mathbf{0} & . & . & \mathbf{0} & I_m & C_{d-1}
\end {pmatrix},
\end{equation}
where $C_0, C_1, \dots , C_{d-1}\in M_m(\Fq)$ and $I_m$ denotes the $m\times m$ identity matrix over $\Fq$ while $\mathbf{0}$ denotes the zero matrix in $M_m(\Fq)$.
\end{definition}
\begin{remark}\label{ssc=bcmi}
It was shown (see the discussion after Conjecture 5.5 in  \cite{m=2} or Appendix A in \cite{split} for an overview) that the Splitting Subspace Theorem is in fact equivalent to the following theorem on block companion matrices.     
\end{remark}

\begin{theorem}
  \label{th:numbcm}
  For any irreducible polynomial $f\in \Fq[x]$ of degree $md$, the number of $(m,d)$-block companion matrices over $\Fq$ having $f$ as their characteristic polynomial equals
  $$
q^{m(m-1)(d-1)}\prod_{i=1}^{m-1}(q^m-q^i).
  $$
\end{theorem}

  It is noteworthy that the problem of counting specific types of block companion matrices having irreducible characteristic polynomial has been considered in other contexts \cite{MR3385053,MR3622678,samtsr} where pseudorandom number generation is of interest. 
We now deduce Theorem \ref{th:numbcm} as a special case of Theorem \ref{th:fibersize} which we prove below, thereby providing an alternate proof of the Splitting Subspace Theorem. 
\begin{definition}
For positive integers $k,\ell$ with $k<\ell$, let $J^{\ell,k}$ denote the $\ell \times k$ matrix given by  
\[J^{\ell,k} := \left[\!\!\begin{array}{c}\bzero\\  I_k\\\end{array}\!\!\right].\]
\end{definition}

\begin{lemma}\label{asc}
The linear matrix polynomial
\[ x \left[\!\!\begin{array}{c} I_k\\ \bzero\\ \end{array}\!\!\right] - J^{\ell,k} \] is unimodular.
\end{lemma}
\begin{proof}
  Since the $k\times k$ minor formed by the last $k$ rows of the above matrix polynomial equals $(-1)^k$ it follows that the GCD of all $k\times k$ minors is 1.
\end{proof}

\begin{definition}
Let $m, \ell$ be positive integers such that $m<\ell$. An $m$-companion matrix of order $\ell$
over $\mFq$ is a square matrix $C$ of the form
\[ C = [J^{\ell,\ell-m}\ A]\] for some $A\in M_{\ell,m}(\mFq)$. We denote the set of all $m$-companion matrices of order $\ell$ over $\Fq$ by $\calM(\ell, m; q)$. Note that
$|\calM(\ell, m; q)| = q^{\ell m}$.  
\end{definition}

Let $\calP(\ell,\mFq)$ denote the set of all monic polynomials of degree $\ell$ over $\mFq$.
Now consider the map $\Phi : \calM(\ell, m; q) \rightarrow \calP(\ell,\mFq)$ given by
\[\Phi(C) := \det(x I_\ell - C).\]
To determine the size of the fibers of $\Phi$, we require a theorem of Wimmer.
\begin{theorem}[Wimmer]
\label{th:wimmer}
Let $F$ be an arbitrary field and let $Y\in M_{\ell,k}(F)$. Suppose $f\in F[x]$ is a monic polynomial of degree $\ell$ and let $f_1(x)\mid \cdots \mid f_k(x)$ be the invariant factors of the polynomial matrix $xI_{\ell,k}-Y$. There exists a matrix $Z \in M_{\ell,\ell-k}(F)$ such that the block matrix $[Y\; Z]$ has characteristic polynomial $f(x)$ if and only if the product $\prod_{i=1}^kf_i(x)$ divides $f(x)$.   
\end{theorem}
\begin{proof}
See Wimmer \cite{Wimmer1974} or Cravo \cite[Thm. 15]{Cravo2009}.  
\end{proof}

\begin{theorem}\label{th:fibersize}
Suppose that $f\in \calP(\ell,\mFq)$ is irreducible. Then 
\[|\Phi^{-1}(f)| = \prod_{t=1}^{m-1} (q^\ell - q^{\ell -t}).\]
\end{theorem}
\begin{proof}
Let $C = [J^{\ell,\ell-m}\ A]\in \calM(\ell,m;q)$ with
$A = [{\bf a}_1\ {\bf a}_2\,\cdots\,{\bf a}_{m-1}\,{\bf a}_m]$, where the ${\bf a}_i$'s are the columns of $A$.
Let $C_0 = J^{\ell,\ell-m}$ and let $C_i = [J^{\ell,\ell-m}\ {\bf a}_1\ {\bf a}_2\,\cdots\,{\bf a}_{i}] $ denote the submatrix of $C$ formed by the first $\ell-m+i$ columns for $1\le i < m$.
Suppose that $\Phi(C) = f$. Since $f$ is irreducible, it follows by Lemma \ref{asc} and Wimmer's theorem that the linear matrix polynomials  
\begin{equation}\label{lmp_ci}
x \left[\!\!\begin{array}{c} I_{\ell -m+i}\\ \bzero\\ \end{array}\!\!\right] - C_{i}
\end{equation}
 are unimodular for $0 \le i\le m-1$. Conversely, if ${\bf a}_1,\ldots, {\bf a}_{m-1}$ are chosen such that the matrix polynomials in \eqref{lmp_ci} are unimodular, then there is a unique choice of ${\bf a}_m$ for which $\Phi(C)=f$. This follows since there are $q^\ell$ total choices for ${\bf a}_m$ and for each monic polynomial $g$ of degree $\ell$, Wimmer's theorem ensures that there exists some choice of ${\bf a}_m$ such that the characteristic polynomial is $g$. By Lemma~\ref{lem:extension} it follows that 
the number of choices for the first $m-1$ columns of $A$ is equal to $\prod_{i=1}^{m-1}(q^\ell - q^{\ell-m+i})$ which proves the result.
\end{proof}

\begin{remark}
In the case where $m$ divides $\ell$, say $d = \ell /m $, the set $\calM(\ell, m; q)$ consists precisely of all $(m,d)$-block companion matrices over $\mFq$. This observation yields the following corollary stated earlier as Theorem \ref{th:numbcm}. 
\begin{corollary}\label{cor:irredfiber}
    For any irreducible polynomial $f\in \Fq[x]$ of degree $md$, the number of $(m,d)$-block companion matrices over $\Fq$ having $f$ as their characteristic polynomial equals
  $$
q^{m(m-1)(d-1)}\prod_{i=1}^{m-1}(q^m-q^{i}).
  $$
  \begin{proof}
    It follows by the above remark that the number of $(m,d)$-block companion matrices over $\Fq$ having $f$ as their characteristic polynomial equals
    $$
\prod_{i=1}^{m-1}(q^{md}-q^{m(d-1)+i})=\prod_{i=1}^{m-1}q^{m(d-1)}(q^m-q^{i}),
$$
which is clearly equal to the given product.
  \end{proof}

\end{corollary}
In light of the above corollary and Remark \ref{ssc=bcmi} we can view Theorem \ref{th:fibersize} as a more general result than the Splitting Subspace Theorem. While our proof relies on results in control theory, it is shorter than the proofs of the theorem appearing in~\cite{sscffa} and~\cite{SKHP2012}.  
\end{remark}

\section{Probability of  Unimodular Polynomial Matrices}\label{sec:unimodular}
We apply Lemma \ref{lem:extension} to positively resolve a conjecture \cite[Conj. 4.1]{zerokernel} concerning the number of unimodular polynomial matrices. For positive integers $d,k,n$ with $k< n$, define  
\[ M_{n,k}(\mFq[x];d) := \left\{ {\bf A} = x^d I_{n,k}
+ \sum_{i=0}^{d-1} x^i A_i\ :
\ A_i \in M_{n,k}(\mFq)\ \mbox{for}\ 0\le i\le d-1\right\}.\]

\begin{theorem}\label{th:density}
The probability that a uniformly random element of $M_{n,k}(\mFq[x];d)$ is unimodular is given by
$\prod_{i=1}^k (1-q^{i-n})$.  
\end{theorem}
\begin{proof}

To each element ${\bf A}$ in $M_{n,k}(\mFq[x];d)$, we associate the corresponding $d$-tuple
of its coefficients $(A_0, A_1,\ldots,A_{d-1}) \in [M_{n,k}(\mFq)]^d$. Now consider the matrix
\begin{align}\label{bform}
B = \left[\begin{array}{ccccc}
\bzero & \bzero & \dots & \bzero & -A_0 \\
I_n & \bzero & \dots & \bzero & -A_1 \\
\vdots & \vdots & \rotatebox{20}{$\ddots$} & \vdots & \vdots \\
\bzero & \bzero & \ldots & I_{n} & -A_{d-1}\\
\end{array}\right]
\end{align}
of dimension $nd \times (nd-n+k)$. Let
\[{\bf B} = x \left[\!\!\begin{array}{c} I_{(d-1)n+k}\\ \bzero\\ \end{array}\!\!\right] - B .\]
By adding $x$ times the $i$\textsuperscript{th} block row to the $(i-1)$\textsuperscript{th} block row successively for $i = d, d-1, \ldots, 2$ in ${\bf B}$ and using suitable column block operations, we obtain
\begin{align*}
{\bf B}' = \left[\begin{array}{ccccc}
\bzero & \bzero & \dots & \bzero & {\bf A} \\
I_n & \bzero & \dots & \bzero & \bzero \\
\vdots & \vdots &  \rotatebox{20}{$\ddots$} & \vdots & \vdots \\
\bzero & \bzero & \ldots & I_{n} & \bzero\\
\end{array}\right],
\end{align*}
where ${\bf A}=x^d I_{n,k}+
\sum_{i=0}^{d-1} x^i A_i \in M_{n,k}(\mFq[x];d)$.
Observe that ${\bf B}$ is equivalent to ${\bf B}'$. So the invariant factors of ${\bf B}$ and ${\bf B}'$ are the same.
Therefore ${\bf B}$ is unimodular if and only if ${\bf A}$ is unimodular. 
By Lemma \ref{lem:extension}, the number of ways to choose the last $k$ columns of the matrix $B$ in \eqref{bform} in such a way that ${\bf B}$ is unimodular is 
$$\prod_{i=1}^{k} (q^{nd}-q^{n(d-1)+i}).$$
On the other hand, the cardinality of $M_{n,k}(\mFq[x];d)$ is clearly $q^{nkd}$ and therefore the probability that a uniformly random element of $M_{n,k}(\mFq[x];d)$ is unimodular is precisely $\prod_{i=1}^k (1-q^{i-n})$.
\end{proof}
Note that the probability computed in the theorem is independent of $d$.
\begin{remark}
The above theorem is a generalization of Corollary~\ref{cor:numsimp} which is evidently the special case $d=1$.   
\end{remark}
Theorem \ref{th:density} parallels a result of Guo and Yang \cite[Thm. 1]{MR3008525} who prove that the natural density of unimodular $n\times k$ matrices over $\Fq[x]$ is precisely $\prod_{i=1}^k (1-q^{i-n})$.    
\begin{remark}
To study the invariant factors of an element ${\bf A} \in M_{n,k}(\mFq[x];d)$, it suffices to study those of the corresponding linear matrix polynomial ${\bf B}$ associated to the matrix $B$ as defined in Equation \eqref{bform}. The matrix polynomial ${\bf B}$ is called the linearization of ${\bf A}$.
\end{remark}

\section*{Acknowledgements}
The third author would like thank Mr. Abhishek Kesarwani and Dr. Santanu Sarkar for useful discussions.

\end{document}